\documentclass[reqno, 12pt]{amsart}
\usepackage[body={7in,9.5in},top=1in, left=0.8in ]{geometry}
\usepackage{enumerate}
\usepackage{enumitem}
\usepackage{cancel}
\usepackage{enumitem}
\usepackage{amssymb}
\newtheorem{theorem}{Theorem}[section]

\newtheorem{lemma}[theorem]{Lemma}

\newtheorem{corollary}[theorem]{Corollary}
\newtheorem{example}[theorem]{Example}

\newtheorem{df}[theorem]{Definition}

\theoremstyle{remark}
\newtheorem{remark}[theorem]{Remark}

\numberwithin{equation}{section}

\begin{document}
\title{On the Gromov product and its generalization in partial metric spaces}


\subjclass[2020]{Primary: 51F99; Secondary: 54E35, 54E99.}

\keywords{Best approximation, Chebyshev set, geodesic metric, Gromov product, hyperconvex metric space, midpoint mapping, partial metric space, radial metric, "river" metric.}

\begin{abstract}
The first aim of this article is to extend the notion of the Gromov product to the class of partial metric spaces. A relation between $\delta$-hyperbolicity of a partial metric space and $\delta$-hyperbolicity of metric spaces endowed with the metrics defined via that partial metric is also established. The second one is to analyze the Gromov product in the wide class of metric spaces in which metrics are defined via Chebyshev sets. Midpoint mappings in this class of spaces will be also examined. 
\end{abstract}

\author[D.Bugajewski]{Dariusz Bugajewski}
\author[K.Jarosi\'nski]{Kacper Jarosi\'nski}

\address[D. Bugajewski]{Department of Nonlinear Analysis and Applied Topology\\
  Faculty of Mathematics and Computer Science\\
  Adam Mickiewicz University\\
  Uniwersytetu Pozna\'nskiego 4\\
  61-614 Pozna\'n\\
  Poland} 
	
\address[K. Jarosi\'nski]{Faculty of Mathematics and Computer Science\\
  Adam Mickiewicz University\\
  Uniwersytetu Pozna\'nskiego 4\\
  61-614 Pozna\'n\\
  Poland}	

\email[D.~Bugajewski]{ddbb@amu.edu.pl}
\email[Kacper Jarosi\'nski]{kacjar1@st.amu.edu.pl}

\maketitle

\section{Introduction} 

In this paper we study one of the basic concepts of hyperbolic geometry, that is, the Gromov product introduced by Gromov in 1987 (see \cite{g}). \\
Firstly, we extend the Gromov product to the class of partial metric spaces introduced by Matthews in \cite{m}. Although this notion was introduced for the needs of computer science (denotational semantics of programming languages), it possesses very interesting pure mathematical properties. An essential difference between metrics and partial metrics is the fact that in the case of partial metric spaces a self-distance of a point to itself does not have to be equal to zero. Further, although these spaces are generally non-Hausdorff, many metric-like tools and concepts can be naturally extended to this setting. In particular, as it was shown in \cite{bmw}, in partial metric spaces compactness and sequential compactness are equivalent. Partial metric spaces have also found interesting applications in other areas, such as the geometry of normed spaces (see~\cite{rosp2}), the domain of words, and complexity spaces (see~\cite{rosp}). Let us emphasize also that the fixed point theory in partial metric spaces seems to be well developed with its similarities and differences in comparison to the classical metric fixed point theory (see e.g. \cite{bm}, \cite{bmw}, \cite{hrs} and \cite{ipr}). Our goal will be to define the Gromov product in the class of partial metric spaces and describe its core properties. \\
It is well-known that to every partial metric space $(U,p)$ one can naturally associate two metrics $p^m$ and $d_m$ (see Section 2 for the definitions). For example, the completeness of a partial metric space $(U,p)$ is equivalent to the completeness of the metric space $(U, p^m)$ or $(U, d_m)$ (see, for example,~\cite[p.~194]{m}). Moreover, one can use these metrics to define some kinds of hyperconvexity in partial metric spaces (see \cite{bko} for more details). Therefore, in Section 3 we compare the Gromov product in a partial metric space with the Gromov products associated with these two metrics. Moreover, we indicate a relation between $\delta$-hyperbolicity of a given partial metric space and $\delta$-hyperbolicities of metric spaces endowed with the metrics defined via that partial metric.\\
Secondly, we are going to analyze the Gromov product in a broad class of metric spaces in which metrics are defined via Chebyshev sets. For that, we will use the construction of metrics presented in the paper \cite{bbp}. Let us recall that the well-known metrics on the plane like the "river" metric or the radial metric are particular cases of that construction. In the first case the Chebyshev set is the $x$-axis while in the second case it is just the singleton $\{(0,0)\}$. Moreover, in connection with the notion of hyperconvexity introduced by Aronszajn and Panitchpakti in \cite{ap} let us recall that if the Chebyshev set generating the metric under consideration is hyperconvex, then the metric defined using it is hyperconvex as well (see \cite{bbp}, Th. 3.1). Our goal will be to give explicit formulae for the Gromov product in this class of spaces. \\
Recall that hyperconvex metric spaces are totally convex and therefore geodesic; moreover, they are complete (see e.g. \cite{ap}, \cite{bbk} or \cite{ek}). It is well known (see e.g. \cite{bbi}, Th. 1.64, p. 19 or \cite{mlg}, p. 5) that in the case of complete metric spaces, geodesicity is equivalent to the existence of a midpoint mapping. In Section 5 we give explicit formulae for a midpoint mapping in metric spaces defined via Chebyshev sets. The formula for a midpoint mapping given in \cite{orv} for the plane with the "river" metric is a special case of that general one. 
 
\section{Preliminaries}

In this section, we collect some key notions and results which will be needed in the sequel. 

Let $K$ be a nonempty subset of a real normed space $X$ and $x\in X$. We say that $y \in K$ is the best approximation to $x$ from $K$ if  
\[
\|x-y\| = inf\{\|x-z\|: z \in K \}. 
\]
The set $K$ is called to have the property $U_{x}$ if the best approximation to $x$ from $K$ is unique and $K$ is said to be a Chebyshev set if it possesses the property $U_{x}$ for every $x\in X$. 

Let $C\subset X$ be a Chebyshev subset of $X$. For any $x\in X$ denote by $x^p$ the best approximation from $x$ to $C$. In the next definition we use the following property of pairs of points $x, y \in X$: 
\begin{equation}\label{cl}
x^p=y^p \quad \mbox{and} \quad x^p, x, y \quad \mbox{are collinear}.
\end{equation} 

\begin{df}\label{cm} 
Let $C\subset X$ be a Chebyshev subset of a normed space $X$ and let $d_C$ be any metric defined on $C$. Define the function $d: X \times X \to [0, +\infty)$ by the formula:
\begin{equation}\label{dcc}
d(x,y) =
 \begin{cases} \|x-y\|, &\!\begin{aligned}
     &\text{if} \quad x, y \quad \text{satisfy the condition} \quad \eqref{cl},\\
 \end{aligned} \\[3ex]
  \|x-x^{p}\| + d_{C}(x^{p},y^{p}) + \|y^{p}-y\|, & \text{otherwise.} \end{cases}
\end{equation} 
\end{df}

It is easy to check that the above defined function $d$ is a metric. Moreover, as it was proven in \cite{bbp} if $(C, d_C)$, where $C$ is a Chebyshev subset of $X$ is hyperconvex, then the metric space $(X, d)$, where $d$ is the metric given in Definition \ref{cm}, is also hyperconvex. 

As it was mentioned in the Introduction, the well-known metrics in $\mathbb R^2$: the "river" metric or the radial metric are particular cases of metrics described in Definition \ref{cm}. For convenience of the reader let us recall the definitions of those two metrics. 

\begin{df}\label{rim} 
The function $d: \mathbb{R}^2\times\mathbb{R}^2\to [0, +\infty)$ defined by the formula:
 \[
d(v_1, v_2) =
\begin{cases}
|y_1 - y_2|, & \text{if } \quad x_1 = x_2, \\
|y_1| + |y_2| + |x_1 - x_2|, & \text{if} \quad x_1 \ne x_2,
\end{cases}
\]
where $v_i = (x_i, y_i) \in \mathbb{R}^2$ for $i = 1, 2$ is said to be the "river" metric.
\end{df}

\begin{df}\label{ram} 
The function $d: \mathbb{R}^2\times\mathbb{R}^2\to [0, +\infty)$ defined by the formula:
  \[
d(v_1, v_2) =
\begin{cases}
\rho(v_1,v_2), & \text{if} \quad  (0,0), v_1, v_2 \text{ are collinear,}\\
\rho(v_1,0)+\rho(v_2,0), & \text{otherwise},
\end{cases}
\]
where $v_i = (x_i, y_i) \in \mathbb{R}^2$ for $i = 1, 2$ and $\rho$ is the Euclidean metric in $\mathbb{R}^2$, is said to be the radial metric.
\end{df} 

The following four definitions need not to be recalled. We place them here just to fix notation. 

\begin{df}\label{gp}
Let $(X,d)$ be a metric space and let $x, y, z \in X$. The Gromov product of $x$ and $y$ at the point $z$ is defined as follows:
\[
(x, y)_{z} = \frac{1}{2}(d(x,z) + d(y,z) - d(x,y)). 
\]
\end{df} 

The point $z$ appearing in the above definition is said to be the reference point or the base point. 

\begin{df}\label{hy}
A metric space $X$ is said to be $\delta$-hyperbolic if for any points $x,y,z,w \in X$ the following "four-points condition" is satisfied, that is  
\[
(x,z)_{w} \geq min\{(x,y)_{w} , (y,z)_{w}\}-\delta. 
\]
\end{df} 

\begin{df}\label{geo} 
A metric space $X$ is said to be geodesic metric space if its any two points can be connected with a geodesic line, that is, for arbitrary $x,y\in X$ there exists an isometric embedding $\gamma\colon [a,b]\to \mathbb R$, $[a,b]\subset \mathbb R$, such that $\gamma(a)=x$ and $\gamma(b)=y$. 
\end{df} 

\begin{df}\label{mm} 
Let $(X,d)$ be any metric space. A mapping $m: X \times X \to X$ is said to be a midpoint mapping for $(X,d)$ if for arbitrary $x,y \in X$ it holds 
\[
d(m(x,y),x)=d(m(x,y),y)=\frac{1}{2}d(x,y).
\] 
In the case when it causes no confusion, we shortly write $m=m(x,y)$. 
\end{df}

At the end of this section we focus on partial metric spaces. 

\begin{df}\label{pm} 
Let $U$ be a nonempty set. A function $p: U\times U \to \mathbb{R_{+}} \cup \{0\}$ is said to be a partial metric if it satisfies the following conditions (for all $x,y,z \in U $): 
\begin{enumerate}
    \item $0\leq p(x,x)\leq p(x,y)$;
    \item if $p(x,x)=p(x,y)=p(y,y)$, then $x=y$;
    \item $p(x,y)=p(y,x)$ (symmetry);
    \item $p(x,z)\leq p(x,y)+p(y,z)-p(y,y)$ (Triangle Inequality). 
\end{enumerate} 
$p(x,x)$ is said to be the size of the point $x$. 
\end{df} 

To every partial metric space $(U,p)$ one can naturally associate two metrics $p^m$ and $d_m$ defined by
\begin{align*}
  p^m(x,y)&:=2p(x,y)-p(x,x)-p(y,y),\\
	d_m(x,y)&:=\max\bigl\{p(x,y)-p(x,x),\ p(x,y)-p(y,y)\bigr\}, 
\end{align*} 
for all $x,y\in U$. 

In what follows we will use the following example of a partial metric space given by Matthews. 

\begin{example}(\cite{m})\label{inte} 
Let $\mathcal I$ be the family of closed bounded intervals in $\mathbb{R_+}\cup\{0\}$, that is, $\mathcal I=\{[a,b] : a,b \in [0,+\infty) \wedge a<b\}$. Then the function $p: \mathcal I\times \mathcal I \to \mathbb{R_{+}} \cup \{0\}$ defined as follows:  
\[
p([a,b],[c,d])=\max\{b,d\}-\min\{a,c\}
\]
for any $a,b,c,d \in [0,+\infty)$ is an example of a partial metric space. 
\end{example} 

\section{The Gromov product in partial metric spaces}

In this section, unless indicated otherwise, $(U,p)$ will denote a partial metric space. We propose the following definition of the Gromov product in partial metric spaces. 

\begin{df}\label{gppm}
The Gromov product $(\cdot,\cdot)_z$, where $z\in U$, is defined as follows:
\[
(x,y)_z=\frac{1}{2}(p(x,z)+p(y,z)-p(z,z)-p(x,y)),
\]
for any $x,y \in U$.
\end{df} 

Analogously to the Definition \ref{hy} we can admit the following 

\begin{df}\label{hy2}
A partial metric space $U$ is said to be $\delta$-hyperbolic if for any points $x,y,z,w \in U$ the following "four-points condition" is satisfied, that is  
\[
(x,z)_{w} \geq min\{(x,y)_{w} , (y,z)_{w}\}-\delta. 
\]
\end{df}
 
The following lemma describes the basic properties of the Gromov product in partial metric spaces. 

\begin{lemma}
Let $x,y,z,u,v \in U$. Then we have:
\begin{enumerate}
\item[(a)] $(x,y)_{z}=(y,x)_{z}$,
\item[(b)] $(x,y)_{x}=(x,y)_{y}=0$,
\item[(c)] $0 \leqslant(x,y)_z \leq \min\{p(x,z),p(y,z)\}$,
\item[(d)] $|(x,y)_{u}-(x,y)_{v}| \leq p(u,v)$,
\item[(e)] $|(x,y)_{u}-(x,z)_{u}| \leq p(y,z)$.
\end{enumerate}
\end{lemma} 

\begin{proof}
The first two properties are obvious. 
\begin{enumerate}
\item[(c)] We have 
\[
(x,y)_z=\frac{1}{2}(p(x,z)+p(y,z)-p(z,z)-p(x,y)) \geqslant \frac{1}{2}(p(x,y)-p(x,y))=0. 
\]
Using the Triangle Inequality we get  
\[
p(y,z)-p(x,y) \leqslant p(x,z)-p(x,x), 
\]
so
\begin{align*}
    (x,y)_z = \frac{1}{2}(p(x,z)+p(y,z)-p(x,y)-p(z,z)) \leqslant p(x,z)-\frac{1}{2}\big(p(x,x)+p(z,z)\big) \leqslant p(x,z).
\end{align*}
Analogously we get $(x,y)_z \leqslant p(y,z)$, so  
\[
(x,y)_z \leqslant \min\{p(x,z),p(y,z)\}.
\] 
\item[(d)] We have to establish the following two inequalities:
\begin{equation}\label{czw}
\begin{split} 
&-p(u,v)\leqslant \frac{1}{2}(p(x,u)+p(y,u)-p(u,u)\\
&-p(x,v)-p(y,v)+p(v,v))
\leqslant p(u,v),
\end{split}
\end{equation}
Using the Triangle Inequality we get 
\begin{align*}
&\frac{1}{2}(p(x,u)+p(y,u)-p(u,u)-p(x,v)-p(y,v)+p(v,v)) \leqslant \\
&\frac{1}{2}(p(x,v)+p(v,u)-p(v,v)+p(y,v)+p(v,u)-p(v,v)\\
&-p(u,u)-p(x,v)-p(y,v)+p(v,v)) =\frac{1}{2}(2p(u,v)-p(v,v)-p(u,u))=\\
&=p(u,v)-\frac{1}{2}(p(v,v)+p(u,u)) \leqslant p(v,u). 
\end{align*} 
Further, we have 
\begin{align*}
&\frac{1}{2}(p(x,v)+p(y,v)-p(v,v)-p(x,u)-p(y,u)+p(u,u)) \leqslant\\
&\frac{1}{2}(p(x,u)+p(u,v)-p(u,u)+p(y,u)+p(u,v)-p(u,u)\\
&-p(v,v)-p(x,u)-p(y,u)+p(u,u)) =\frac{1}{2}(2p(u,v)-p(v,v)-p(u,u))=\\
&=p(u,v)-\frac{1}{2}(p(v,v)+p(u,u)) \leqslant p(v,u), 
\end{align*}
which ends the proof. 
\item[(e)] This time we have to establish the following two inequalities:
\begin{equation}\label{pi} 
-p(y,z)\leqslant \frac{1}{2}(p(y,u)-p(x,y)-p(z,u)+p(x,z))\leqslant p(y,z),
\end{equation}
By the Triangle Inequality, we get 
\begin{align*}
& \frac{1}{2}(p(y,u)+p(x,z)-p(x,y)-p(z,u))\leqslant\\
&\frac{1}{2}(p(y,z)+p(z,u)+p(x,y)+p(y,z)-p(x,y)-p(z,u))=p(y,z).
\end{align*}
Further, we have 
\begin{align*}
& \frac{1}{2}(p(x,y)+p(z,u)-p(y,u)-p(x,z))\leqslant\\
&\frac{1}{2}(p(x,z)+p(z,y)+p(z,y)+p(y,u)-p(y,u)-p(x,z))=p(y,z), 
\end{align*}
which ends the proof. 
\end{enumerate}
\end{proof} 

Now, we compare the Gromov product in a partial metric space $(U,p)$ with the Gromov products associated with the metrics $p^m$ and $d_m$. Denote by $(x,y)_z$ the Gromov product associated with the partial metric $p$, by $(x,y)_z^{p^m}$ - the Gromov product in the metric $p^m$, and by $(x,y)_z^{d_m}$ - the Gromov product in the metric $d_m$, where $x,y,z \in U$.  

\begin{lemma}
For any $x,y,z \in U$, it holds:
\begin{enumerate}
\item[(a)]
\[
(x,y)_z^{p^m}=2(x,y)_z;
\]
\item[(b)]
\[
(x,y)_z^{d_m}= \begin{cases}
    (x,y)_z+\frac{1}{2}(p(x,x)-p(z,z)), & \text{if }\quad p(z,z)\leqslant p(x,x) \leqslant p(y,y), \\
    (x,y)_z+\frac{1}{2}(p(z,z)-p(x,x)), & \text{if} \quad p(y,y) \leqslant p(x,x) \leqslant p(z,z),
\end{cases}
\]
\[
(x,y)_z^{d_m}= \begin{cases}
    (x,y)_z+\frac{1}{2}(p(y,y)-p(z,z)), & \text{if } \quad p(z,z)\leqslant p(y,y) \leqslant p(x,x), \\
    (x,y)_z+\frac{1}{2}(p(z,z)-p(y,y)), & \text{if } \quad  p(x,x) \leqslant p(y,y) \leqslant p(z,z),
\end{cases}
\]
\[
(x,y)_z^{d_m}=(x,y)_z \ \text{if } \ p(y,y) \leqslant p(z,z) \leqslant p(x,x) \ \text{or} \ p(x,x) \leqslant p(z,z) \leqslant p(y,y);  
\]
\item[(c)]
\[
(x,y)_z\leqslant (x,y)_z^{d_m} \leqslant (x,y)_z +\frac{1}{2} \max\{|p(z,z)-p(x,x)|, |p(z,z)-p(y,y)|\}.
\]
and 
\[
\frac{1}{2} (x,y)_z^{p^m}\leqslant (x,y)_z^{d_m} \leqslant \frac{1}{2}\left ((x,y)_z^{p^m}+ \max\{|p(z,z)-p(x,x)|, |p(z,z)-p(y,y)|\}\right). 
\] 
\end{enumerate}
\end{lemma} 

\begin{proof}
\begin{enumerate} 
\item[(a)]
For all $x,y,z \in U$ we have 
    \begin{align*}
    &(x,y)_z^{p^m}=\frac{1}{2}(p^m(x,z)+p^m(y,z)-p^m(x,y))= \\
    &=\frac{1}{2}(2p(x,z)+2p(y,z)-2p(x,y)-2p(z,z))=2(x,y)_z.
    \end{align*} 
		\item[(b)]
For all $x,y,z \in U$ we have 
\begin{align*}
&(x,y)_z^{d_m}=\frac{1}{2}\big(d_m(x,z)+d_m(y,z)-d_m(x,y)\big)=\\
&\frac{1}{2}\big(\max\{p(x,z)-p(x,x), \ p(x,z)-p(z,z)\}+\max\{p(y,z)-p(y,y), \ p(y,z)-p(z,z)\}\\
&-\max\{p(x,y)-p(x,x), \ p(x,y)-p(y,y)\}\big).
\end{align*} 
Let us consider suitable subcases. 
\begin{enumerate}[label=(\roman*)]
\item If $p(z,z)\leqslant p(x,x) \leqslant p(y,y)$, then  
\begin{align*}
    &(x,y)_z^{d_m}=\frac{1}{2}(p(x,z)-p(z,z)+p(y,z)-p(z,z)-p(x,y)+p(x,x))=\\
    &=\frac{1}{2}(p(x,z)+p(y,z)-p(x,y)-p(z,z)+p(x,x)-p(z,z))= \\&=(x,y)_z+\frac{1}{2}(p(x,x)-p(z,z)).
\end{align*}
\item If $p(x,x)\leqslant p(y,y) \leqslant p(z,z)$, then  
\begin{align*}
    &(x,y)_z^{d_m}=\frac{1}{2}(p(x,z)-p(x,x)+p(y,z)-p(y,y)-p(x,y)+p(x,x))=\\
    &=\frac{1}{2}(p(x,z)+p(y,z)-p(x,y)-p(z,z)+p(z,z)-p(y,y))= \\&=(x,y)_z+\frac{1}{2}(p(z,z)-p(y,y)).
\end{align*}
\item If $p(x,x)\leqslant p(z,z) \leqslant p(y,y)$, then 
\begin{align*}
    &(x,y)_z^{d_m}=\frac{1}{2}(p(x,z)-p(x,x)+p(y,z)-p(z,z)-p(x,y)+p(x,x))=\\
    &=\frac{1}{2}(p(x,z)+p(y,z)-p(x,y)-p(z,z))= \\&=(x,y)_z.
\end{align*} 
The other three subcases are analogous to the above ones. 
\item[(c)] It is an obvious consequence of items (a) and (b). 
\end{enumerate}
\end{enumerate}
\end{proof} 

Using the above lemma one can infer about $\delta$-hyperbolicity of metric spaces in which the metrics are defined via a given partial metric.  

\begin{corollary}\label{hyb}
Let $(U,p)$ be a $\delta$-hyperbolic partial metric space. Then: 
\begin{enumerate}
\item[(a)] $(U,p^m)$ is $2\delta$-hyperbolic;
\item[(b)] assuming additionally that the sizes of all points $x\in U$ are bounded, then $(U,d_m)$ is $s+\delta$-hyperbolic, where $s=\sup\limits_{x\in U} p(x,x)$. 
\end{enumerate}
\end{corollary}

\begin{proof} 
It is an obvious consequence of the above lemma.
\end{proof}
Now, we illustrate the Gromov product in partial metric spaces considerning the following 

\begin{example}\label{ipm}
Let $(\mathcal I,p)$ be the partial metric space defined in Example \ref{inte}. Then the Gromov product can be described by the following formulae: 
\begin{align*}   
&([a,b],[c,d])_{[e,f]}=\frac{1}{2}(p([a,b],[e,f])+p([c,d],[e,f])+\\
&-p([a,b],[c,d])-p([e,f],[e,f]))=\\
&=\text{max}\{b,f\}-\text{min}\{a,e\}+\text{max}\{d,f\}-\text{min}\{c,e\}+\\
&-\text{max}\{b,d\}+\text{min}\{a,c\}-f+e.
\end{align*}
For example, let us consider the intervals in the following setting:
\begin{equation} \label{nr1}
[e,f] \subset [c,d] \subset [a,b]. 
\end{equation}
Then 
\[
([a,b],[c,d])_{[e,f]}=\frac{1}{2}(b-a+d-c-b+a-f+e)=\frac{1}{2}(e-c+d-f).
\]
Let us notice that in this case the Gromov product depends only on the intervals $[c,d]$ and $[e,f]$.\\
Now, let us consider the intervals in the following setting: 
\begin{equation}\label{nr2} 
[c,d] \subset [e,f] \subset [a,b] ,
\end{equation}
which means that the base point changes essentially. Then 
\[
([a,b],[c,d])_{[e,f]}=\frac{1}{2}(b-a+f-e-b+a-f+e)=0.
\]
\end{example} 

\section{The Gromov product in metric spaces involving Chebyshev sets}

In the next proposition, not losing the main idea, we will assume that the set $C$ appearing in the Definition \ref{cm} is the closed ball centered at zero and contained in a normed space $X$. 

\begin{theorem}
The Gromov product defined in the metric space described in the Definition \ref{cm} can be expressed by the following formulae: 
\begin{enumerate}   
\item[(a)] if the points $x^p,y^p,z^p$ are different, then 
\[
(x,y)_z=\|z-z^p\|+(x^p,y^p)^{C}_{z^p},
\]
where $(x^p,y^p)^{C}_{z^p}$ denotes the Gromov product defined via the metric on the Chebyshev set $C$; 
\item[(b)] if $x^p=z^p$ and $y^p\neq x^p$, then 
\[
(x,y)_z=\|x-z\|,
\] 
if $\|z-z^p\|>\|x-x^p\|$; otherwise it is equal to $0$;
\item[(c)] if $y^p=z^p$ and $y^p\neq x^p$, then analogously to the previous case 
\[
(x,y)_z=\|y-z\|,
\]
if $\|z-z^p\|>\|y-y^p\|$; otherwise it is equal to $0$;
\item[(d)] if $x^p=y^p$ and $z^p\neq x^p$, then 
\[
(x,y)_z=\|z-z^p\|+d_C(x^p,z^p)+\|x-x^p\|=d(x,z),
\]
if $\|y-y^p\|>\|x-x^p\|$; otherwise 
\[
(x,y)_z=\|z-z^p\|+d_C(y^p,z^p)+\|y-y^p\|=d(y,z).
\] 
\item[(e)] if $x^p=y^p=z^p$, then
\begin{equation*}
(x,y)_z=\frac{1}{2}(\|x-z\|+\|y-z\|-\|x-y\|).
\end{equation*}
\end{enumerate}
\end{theorem}

\begin{proof}
\begin{enumerate}   
\item[(a)] We have 
\begin{align*}
    (x,y)_z
    &=\frac{1}{2}(d(x,z)+d(y,z)-d(x,y))=\\
    &=\frac{1}{2}(\|x-x^p\|+d_C(x^p,z^p)+\|z-z^p\|+\|y-y^p\|+d_C(y^p,z^p)+\|z-z^p\|\\
    &-\|x-x^p\|-d_C(x^p,y^p)-\|y-y^p\|)=\\
    &=\|z-z^p\|+\frac{1}{2}(d_C(x^p,z^p)+d_C(y^p,z^p)-d_C(x^p,y^p))
    =\|z-z^p\|+(x^p,y^p)^{C}_{z^p}.
    \end{align*}
\item[(b)] We have 
\begin{align*}
    (x,y)_z
    &=\frac{1}{2}(\|x-z\|+\|y-y^p\|+d_C(y^p,z^p)+\|z-z^p\|\\
    &-\|y-y^p\|-d_C(y^p,x^p)-\|x-x^p\|)\\
		&=\frac{1}{2}(\|x-z\|+\|z-z^p\|-\|x-z^p\|).
    \end{align*}
Now, there are two possibilities. First, if $\|z-z^p\|>\|x-z^p\|$, then:
\[
(x,y)_z=\|x-z\|.
\]
Otherwise $(x,y)_z=0$. 
\item[(c)] The proof of this case is analogous to the proof of case (b). 	
\item[(d)] We have 
    \begin{align*}
    (x,y)_z
    &=\frac{1}{2}(\|x-x^p\|+d_C(x^p,z^p)+\|z-z^p\|\\
    &+\|y-y^p\|+d_C(y^p,z^p)+\|z-z^p\|-\|x-y\|)=\\
    &=\|z-z^p\|+d_C(x^p,z^p)+\frac{1}{2}(\|x-x^p\|+\|y-y^p\|-\|x-y\|).
    \end{align*}
Arguing similarly as in the case (b) we infer that if $\|y-y^p\|>\|x-x^p\|$, then 
    \[
    (x,y)_z=\|z-z^p\|+d_C(x^p,z^p)+\|x-x^p\|=d(x,z);
    \]
otherwise, we get: 
    \[
    (x,y)_z=\|z-z^p\|+d_C(x^p,z^p)+\|y-y^p\|=d(y,z).
    \]
\item[(e)] We have 
\[
(x,y)_z=\frac{1}{2}(d(x,z)+d(y,z)-d(x,y))=\frac{1}{2}(\|x-z\|+\|y-z\|-\|x-y\|)
\]
Because $x^p=y^p=z^p$, so:
\begin{itemize}
\item $(x,y)_z=\|x-z\|$, if $x$ is between $y$ and $z$,
\item $(x,y)_z=\|y-z\|$, $y$ is between $x$ and $z$,
\item $(x,y)_z=0$, $z$ is between $x$ and $y$.
\end{itemize}
\end{enumerate}
\end{proof}

\begin{remark}
Let us notice that in the case of the radial metric on the plane, $C=\{(0,0)\}$. Obviously, in the case $x^p=y^p=z^p$, however despite of that fact, one should consider the issue of collinearity of points ${x,y,z}$. We will illustrate it in the next example. 
\end{remark}

\begin{example}
Let us consider $\mathbb{R}^2$ with the radial metric and $C=\{(0,0)\}$. In this situation let us consider a particular case, when $x,x^p,z\in \mathbb{R}^2$ are not collinear, $y,y^p,z\in \mathbb{R}^2$ are not collinear and $x,x^p,y\in \mathbb{R}^2$ are also not collinear. Then we have:
\begin{align*}
    &(x,y)_z= \frac{1}{2}(\rho(x,0)+\rho(z,0)+\rho(y,0)+\rho(z,0)-\rho(x,0)-\rho(y,0))=\\
    &= \rho(z,0),
\end{align*}
Similarly, $(x,z)_y=\rho(y,0)$ and $(y,z)_x=\rho(x,0)$. Hence, in this particular case the Gromov product $(x,y)_z$ is equal to the Euclidean distance of the reference point to the singleton $\{(0,0)\}$.
\end{example} 

Now, let us consider a particular example in which the set $C$ appearing in the Definition \ref{cm} is not a closed ball centered at zero and contained in a normed space $X$

\begin{example}
Let us consider the space $(\mathbb{R}^2,d)$ with the "river" metric and $C=\{(x,y) \in \mathbb{R}^2: x\in\mathbb{R},\ y=0\}$. Let us consider a particular case in which $x,x^p,z\in \mathbb{R}^2$ are collinear and $y,y^p,z\in \mathbb{R}^2$ are not collinear. Then the points $x,x^p,y\in \mathbb{R}^2$ are not collinear, too. Denoting $x=(x_1,x_2), y =(y_1,y_2), z =(z_1,z_2), x^p=z^p=(x^p_1,x^p_2), y^p=(y^p_1,y^p_2)$ and assuming additionally that $x_2>z_2>0$, we get 
\begin{align*}
    (x,y)_z= \frac{1}{2}(|x_2-z_2|+|z_2|+|x^p_1-y_1^p|+|y_2|-|x_2|-|x_1^p-y_1^p|-|y_2|)= 0,
\end{align*}
\begin{align*}
    (x,z)_y= \frac{1}{2}(|x_2|+|x^p_1-y_1^p|+|y_2|+|z_2|+|x_1^p-y_1^p|+|y_2|-|x_2-z_2|)= \rho(y,z),
\end{align*}
\begin{align*}
    (y,z)_x= \frac{1}{2}(|x_2|+|x^p_1-y_1^p|+|y_2|+|x_2-z_2|-|z_2|-|x_1^p-y_1^p|-|y_2|)= d(x,z).
\end{align*}
\end{example} 

\section{Other geometric properties of metrics involving Chebyshev sets}

In this section we consider metric spaces introduced in Definition \ref{cm}. The following result shows how a midpoint mapping can be defined in this general case. 

\begin{theorem}
Let $(X,d)$ be a metric space as in Definition \ref{cm} and let $x,y \in X$. A midpoint mapping for such a metric space can be defined as follows: 
\begin{itemize}
\item[(a)] if points $x, x^p,y$ are collinear and $x^p=y^p$, then 
\begin{align*}
m(x,y)=\frac{x+y}{2},
\end{align*}
\item[(b)] if points $x,x^p,y$ are not collinear, then 
\begin{equation}\label{xw}
m(x,y)=x+\frac{d(x,y)}{2\|x-x^p\|}(x^p-x),
\end{equation}
under the assumption that $\|x-x^p\|\geqslant \frac{1}2{(d(x,y))}$, or 
\begin{equation}\label{yw}
m(x,y)=y+\frac{d(x,y)}{2\|y-y^p\|}(y^p-y),
\end{equation}
under the assumption that $\|y-y^p\|\geqslant \frac{1}2{(d(x,y))}$;
in the case when $\|x-x^p\|<\frac{1}2{(d(x,y))}$ and $\|y-y^p\|<\frac{1}2{(d(x,y))}$, a point $m(x,y)\in C$ and it satisfies the following equality: 
\begin{equation}\label{mi}
(x^p,y^p)_{m(x,y)}^C=0.
\end{equation}
Analogously in the case when points $x,x^p,y$ are collinear and $x^p \neq y^p$.
\end{itemize}
\end{theorem} 

\begin{proof} 
\begin{itemize}
\item[(a)]
In this case it is clear that $m(x,y)$ is a midpoint of the segment joining the points $x,y$, that is $m(x,y)=\frac{x+y}{2}$. Note the $m(x,y)$ is uniquely determined, because if were $m^p \neq x^p$, then we would have:       
\[
\begin{cases}
||m-m^p||+d_C(m^p,x^p)+||x-x^p||=\frac{1}{2}||x-y||\\
||m-m^p||+d_C(m^p,y^p)+||y-y^p||=\frac{1}{2}||x-y||
\end{cases}
\]
so 
\[
||x-x^p||=||y-x^p||.
\]
It would imply that $x=y$, which is not possible, because $x\neq y$. Thus $x^p=m^p$ and 
\[
\begin{cases}
||m-x||=\frac{1}{2}||x-y||, \\
||m-y||=\frac{1}{2}||x-y||. 
\end{cases}
\]
Hence 
\[
||x-y||=||m-x||+||m-y|| 
\]
and the point $m$ is the midpoint of the segment joining the points $x$ and $y$. 
\item[(b)]
Assume now that the points $x,x^p,y$ are not collinear and $\|x-x^p\|\geqslant \frac{1}{2}d(x,y)$. Let us notice that a point $m$ has to belong to the segment joining the points $x$ and $x^p$, because otherwise we would have $x^p \neq m^p$ and 
\[
\begin{cases}
\|m-m^p\|+d_C(m^p,x^p)+\|x-x^p\|=\frac{1}{2}(\|x-x^p\|+d_C(x^p,y^p)+\|y-y^p\|)\\
\|m-m^p\|+d_C(m^p,y^p)+\|y-y^p\|=\frac{1}{2}(\|x-x^p\|+d_C(x^p,y^p)+\|y-y^p\|).\\
\end{cases}
\]
Then 
\[
d_C(x^p,y^p)=d_C(m^p,x^p)+d_C(m^p,y^p)+2\|m-m^p\| \geqslant d_C(x^p, y^p)+2\|m-m^p\|,
\]
so
\[
2\|m-m^p\| \leqslant0 \ \implies \ m=m^p. 
\]
Because 
\[
\|x-x^p\|\geqslant \frac{1}{2}d(x,y)=d(m,x)=\|m-m^p\|+d_C(x^p,m^p)+||x-x^p\|,
\]
so 
\[
d_C(x^p,m^p)=0 \ \implies \ x^p = m^p,
\]
what gives a contradiction. 
The segment $[x,x^p]\subset X$ can be represented as the set of points:
\[
x+t(x^p-x),
\]
where $t\in [0,1]$. Since $m\in [x,x^p]$, the parameter $t$ uniquely determines $m$. Because $d(m,x)=\frac{1}{2}d(x,y)$, we get  
\[
\|m-x\|=\|t(x^p-x)\|=\frac{1}{2}d(x,y),
\]
so
\[
t=\frac{d(x,y)}{2\|x-x^p\|}\in[0,1].
\]
Hence 
\[
m(x,y)=x+\frac{d(x,y)}{2\|x-x^p\|}(x^p-x)
\] 
and
\begin{align*}
&d(m(x,y),y)=d(x+\frac{d(x,y)}{2\|x-x^p\|}(x^p-x),y)=\\
&=\left(1-\frac{d(x,y)}{2\|x-x^p\|}\right)\|x-x^p\|+d_C(x^p,y^p)+\|y-y^p\|=\\
&=\frac{\|x-x^p\|-d_C(x^p,y^p)-\|y-y^p\|}{2\|x-x^p\|}\|x-x^p\|+d_C(x^p,y^p)+\|y-y^p\|=\\
&\frac{1}{2}d(x,y). 
\end{align*}
In the case when $\|y-y^p\|\geqslant \frac{1}{2}d(x,y)$ the proof is analogous.\\
Now, consider the case when $\|x-x^p\|<\frac{1}{2}d(x,y)$ and $\|y-y^p\|<\frac{1}{2}d(x,y)$. Then $m \in C$ and $m=m^p$. Because 
\[
\begin{cases}
\|m-m^p\|+d_C(m^p,x^p)+\|x-x^p\|=\frac{1}{2}d(x,y)\\
\|m-m^p\|+d_C(m^p,y^p)+\|y-y^p\|=\frac{1}{2}d(x,y),\\
\end{cases}
\]
so
\[
\begin{cases}
d_C(m,x^p)=\frac{1}{2}d(x,y)-\|x-x^p\|\\
d_C(m,y^p)=\frac{1}{2}d(x,y)-\|y-y^p\|.\\
\end{cases}
\]
Therefore 
\[
d_C(m,x^p)+d_C(m,y^p)=d_C(x^p,y^p),
\]
which means that \eqref{mi} holds. 
\end{itemize}
\end{proof} 

To illustrate the above result, let us consider $\mathbb R^2$ with the radial metric; then obviously $C=\{(0,0)\}$ and $x^p=y^p$. Immediately from the above result we get the following 

\begin{corollary}
Let points $v_{1}=(x_{1},y_{1}),v_{2}=(x_{2},y_{2})\in \mathbb R^2$ with $x_{1},y_{1},x_{2},y_{2}>0$ be given. A midpoint mapping for the radial metric on $\mathbb R_+^2$ can be defined as follows:
\begin{itemize}
\item[(a)] if the points $v_{1}, v_{2}, (0,0)$ are collinear, then:
\begin{align*}
m(v_{1},v_{2})=\Big(\frac{x_{1}+x_{2}}{2},\frac{y_{1}+y_{2}}{2}\Big);
\end{align*}
\item[(b)] if the points $v_{1}, v_{2}, (0,0)$ are not collinear, then:
\begin{equation}\label{sr}
m(v_{1},v_{2})=\Bigg(\frac{x_2}{2}\Bigg(1-\frac{\rho(v_1,0)}{\rho(v_2,0)}\Bigg), \frac{y_2}{2}\Bigg(1-\frac{\rho(v_1,0)}{\rho(v_2,0)}\Bigg)\Bigg),
\end{equation}
under the assumption that $\rho(v_1,0)<\rho(v_2,0)$ (analogously if $\rho(v_2,0)<\rho(v_1,0)$). \\
In the case when $\rho(v_1,0)=\rho(v_2,0)$:
\[
m(v_{1},v_{2})=(0,0).
\]
\end{itemize}
\end{corollary}

\end{document}